\documentclass[10pt]{amsart}
\usepackage{mathrsfs}
\usepackage{a4wide}
\usepackage{amsmath,amssymb,amsfonts,amsthm, mathtools, changes}

\usepackage{color}
\definecolor{dgreen}{RGB}{0,126,0}
\usepackage{soul}

\theoremstyle{plain}

\newtheorem{theorem}{Theorem}[section]
\newtheorem{lemma}[theorem]{Lemma}
\newtheorem{proposition}[theorem]{Proposition}
\newtheorem{corollary}[theorem]{Corollary}
\newtheorem{remark}[theorem]{Remark}

\theoremstyle{definition}
\numberwithin{equation}{section}

\def\pa{\partial}
\def\til{~}
\def\bary{\mathrm{bar}}
\def\beq{\begin{equation}}
\def\eeq{\end{equation}}

\def\de{\delta}
\def\R{\mathbb{R}}

\def\Z{\mathbb{Z}}
\def\N{\mathbb{N}}
\def\dist{\textup{dist}}

\def\H{\mathcal{H}}
\def\D{\mathcal{D}}
\def\P{P}

\def\ov{\overline}

\renewcommand{\div}{\mathrm{div}}

\newcommand{\f}{\varphi}
\newcommand{\e}{\varepsilon}

\newcommand{\medint}{-\kern -,375cm\int}
\newcommand{\medintinrigo}{-\kern -,315cm\int}

\def\beq{\begin{equation}}
\def\eeq{\end{equation}}

 \title{Long time behaviour of discrete volume preserving\\ mean curvature flows}

 \author[M. Morini]
 {Massimiliano Morini}
 \address[Massimiliano Morini]{Dip.~di Scienze Matematiche Fisiche e Informatiche, Univ.~Parma, Italy}
 \email[M. Morini]{massimiliano.morini@unipr.it}

 \author[M. Ponsiglione]
 {Marcello Ponsiglione}
 \address[Marcello Ponsiglione]{Dip.~di Matematica,
 Univ.~Roma-I ``La Sapienza'', Roma, Italy}
 \email[M. Ponsiglione]{ponsigli@mat.uniroma1.it}

 \author[E. Spadaro]
 {Emanuele Spadaro}
 \address[Emanuele Spadaro]{Dip.~di Matematica,
 Univ.~Roma-I ``La Sapienza'', Roma, Italy}
 \email[E. Spadaro]{spadaro@mat.uniroma1.it}


\begin{document}

\begin{abstract}
In this paper we analyse the Euler implicit scheme for the volume preserving mean curvature flow.
We prove the exponential convergence of the scheme to a finite union of disjoint balls with equal volume for any bounded initial set with finite perimeter. 
\end{abstract}

\maketitle

 \bibliographystyle{plain}

\tableofcontents

\section*{Introduction}
We study the asymptotic behaviour of a discrete-in-time approximation of the mean curvature flow with constrained volume.
This geometric evolution consists in a family of evolving sets $ [0,+\infty)\ni 
t\mapsto E_t\subset\R^N$, whose normal velocities $V(x,t)$ at any point $x\in \partial E_t$ are proportional to the scalar mean curvature $H_{\pa E_t}(x)$ of $\partial E_t$ at $x$ corrected by a constant forcing term which guarantees that the measure of $|E_t|$ is preserved during the flow:
\begin{equation}\label{e:equazione modello}
V(x,t) = \overline  H_{\pa E_t} - H_{\pa E_t}(x), \qquad \overline  H_{\pa E_t} := \frac{1}{\mathcal{H}^{N-1}(\partial E_t)}\int_{\partial E_t} H_{\pa E_t}(x) \, d\mathcal{H}^{N-1}(x) .
\end{equation}
Under suitable assumptions, this geometric flow is a model for coarsening phenomena in physical systems.
For example, one can consider mixtures that,
after a first relaxation time, can be described by two subdomains of nearly pure phases far from equilibrium,
evolving in a way to minimize the total interfacial area between the phases while keeping their volume constant 
(further details on the physical background can be found in \cite{CRCT95, MuSe13,TW72,W61}).

It is well-known that a typical evolution of \eqref{e:equazione modello} develops singularities of different kinds in finite time: components shrinking to points and disappearance, collisions and merging of domains, pinch-offs etc{\ldots} 
Compared with the more familiar unconstrained mean curvature flow, the possible changes in the topology of the evolving sets $E_t$ of \eqref{e:equazione modello} are  even wider, because of the nonlocal character of the flow and the subsequent lack of comparison principles. There exist singular solutions also in the two dimensional case (see \cite{M, MS}). 

It is then clear that the analysis of the long time behaviour of systems exhibiting coarsening requires the introduction of suitable notions of weak solutions which allow the formation of singularities and extend the flow past the singular times.
This is a well-established feature of curvature flows, and many definitions of weak solutions have been introduced in the literature.
Here, we follow the method proposed independently by Almgren, Taylor and Wang \cite{ATW} and by Luckhaus and Sturzenhecker \cite{LS}, based on De Giorgi's minimizing movement approach.
The authors consider an implicit time-discretization of the flow, which is regarded as a gradient flow of the perimeter functional with respect to a metric resembling the $L^2$-distance. 
The limiting time-continuous flow constructed with this method is usually referred to as flat flow.
This approach has the advantage to be easily adapted to volume constrained mean curvature flow as shown in \cite{MSS}, producing global-in-time solutions which then permit to analyse the equilibrium configurations reached in the long time asymptotics.

Previous results on the long time behaviour of the volume constrained mean curvature flow are mostly confined to the case of smooth solutions, starting from specific classes of initial regular sets ensuring the existence of global regular solutions; for example  uniformly convex and nearly spherical initial sets (see \cite{H, ES}), or nearly strictly stable initial sets in the three and four dimensional flat torus (see \cite{Nii}).  For more general initial data, 
the long time behaviour in the context of flat flows of convex and star-shaped sets (see \cite{BCCN, KK})  has been characterized only up to (possibly diverging in the case of \cite{BCCN}) translations.

In this paper we characterize the long-time limits of the discrite-in-time approximate flows constructed by the Euler implicit scheme introduced in \cite{ATW, LS} plus the volume constraint.
Our main result in Theorem \ref{mainthm} establishes that for every initial bounded set $E_0$ with finite perimeter, any discrete-in-time mean curvature flow with volume constraint  converges exponentially fast to a finite union of disjoint balls with equal radii.

\subsection{Exponential convergence of dissipations}
The main issue here is to prove that the evolving sets remain uniformly bounded for all times and that the long-time limit is truly unique (and not just up to translations).
 This is a non-trivial fact, since the discrete scheme (as well as its continuous counterpart) does not satisfies comparison principles. To the best of our knowledge, the boundedness of the flow was not known previously for any weak notion of the nonlocal mean curvature flow.
The problem is that, summing the dissipations of the discrete schemes formally provides a $L^2$ (with respect to time) bound on the velocity of the evolving set, rather than  the necessary $L^1$ bound: roughly speaking, one needs to sum the square roots of the dissipations rather than the dissipations themself.

By a compactness argument and a characterization of the stationary configurations of the discrete flow (based on a recent theorem by Delgadino and Maggi \cite{DM} on Caccioppoli sets with constant weak mean curvature), the connected components of the evolving sets converge up to translations to balls. 

Therefore, in order to improve the estimate on the dissipations, 
we compare at each time step the energy of the actual minimizers with that of a properly placed union of balls (one nearby each connected component of the evolving sets).
In this way, it turns out to be necessary to 
estimate the dissipation of a nearly spherical set with respect to a close-by ball (this condition is  definitively reached by each connected component of the flow thanks to the compactness argument above) with the actual dissipation of the flow.
Indeed, if this happens, one can control the sum of the dissipations for all times larger than any time $t$ from above by the energy dissipated at time $t$, thus clearly leading to an exponential decay of the dissipations and in turn of the velocity of the evolvings sets. 

The comparison of the dissipation of the flow and the one with respect to balls is equivalent to prove a control of the $L^2$-norm of the function parametrizing a nearly spherical set with the $L^2$-norm of the oscillation of its mean curvature.
In the case of zero oscillation, this is nothing but Alexandrov's theorem on the characterization of the spheres as the unique embedded hypersurfaces with constant mean curvature and thus the aforementioned estimate  may be regarded as a quantitative  version of it.
Its proof in turn relies on a result proven by Krummel and Maggi \cite[Theorem~1.10]{KM}, which can be seen as a sort of higher order Fuglede-type estimate involving the first variation of the perimeter rather than the perimeter itself.
For different quatitative versions of Alexandrov's theorem see also \cite{CV, CM17, D+18}.

\subsection{Comments and open questions}
It remains an open question to extend the results of the present paper to the flat flows, that is to obtain uniform estimates as the time step converges to zero.
We conjecture the result to hold true, but the compactness arguments used in this paper in order to gain the closeness to sphere is not quantitative and does not allow to pass to the limit.
Moreover, it is worth mentioning that, both nonlocal flat flows and discrete-in-time nonlocal flows are not uniquely defined. There are different ways to impose the volume constraint (by constraining the volume  at each time-step, or by penalization, or by tuning suitably a forcing term). The flows so defined can be different (also because of the congenital non-uniqueness of mean curvature flows), but we expect  that the asymptotics are always finite unions of equal volume balls.

Finally, we stress that the main Theorem \ref{mainthm} is sharp for what concerns the limiting configurations. Indeed, we can show that the union of suitably distant balls with equal radii is a stationary solution for the discrete flow and, hence, a possible asymptotics. However, such solutions are clearly unstable, because whatever small increase in one of the radii (and consequent decrease in the others) will drive the evolution towards a different asymptotic limit. This is easily seen for two disjoint balls with equal radii: any configuration of concentric balls with the same total volume but slightly perturbed radii will have strictly lower perimeter  and thus will converge to  a single ball (see Corollary~\ref{cor:menodidue}). 
In this regard, we conjecture that, generically, the evolution of an initial set $E_0$ will converge to a {\it unique} limiting ball with the same volume.

\subsection{Structure of the paper}
In the first section we prove the Fuglede-type inequality for nearly spherical sets with almost constant mean curvature. Section 2 is then devoted to the introduction of the incremental problem at the basis of the discrete-in-time mean curvature flow.
Finally, in the last section of the paper we prove the exponential convergence of the discrete flow to a union of well-separated balls with the equal volume.
 
\section{A quantitative Alexandrov Theorem for nearly spherical sets}
 In this section we prove a variant of a stability inequality by Krummel and Maggi \cite[Theorem 1.10]{KM} related to Alexandrov's theorem, whose statement we recall for the reader's convenience. To this aim, and for later use, we fix the following notation:
$B$ denotes the unit ball of $\R^N$.  Given ${m}>0$, we denote by $B^{(m)}$ the ball centered at the origin with volume $m$ and we denote by $r(m)$ its radius $\Big(\frac{m}{\omega_N}\Big)^{\frac1N}$.
For  any  measurable function $\varphi:\partial B^{(m)} \to (-1, +\infty)$, we denote by $E_{\f,m}$ the set in $\R^N$ whose boundary is described through the radial graph of $\f$, namely   
\beq\label{schifo}
 E_{\f, m}:= \big\{t x \in \R^N : \, x\in\partial B^{(m)},\; 0\le t  \le 1+ \varphi(x)\big\}.
\eeq
In the case $m=\omega_N$ we will write $B$ instead of $B^{(\omega_N)}$ (to denote the unit ball) and  $ E_{\f}$ instead of $ E_{\f, \omega_N}$.
If the parametrizing function $\f$ has small $L^\infty$-norm, then the set $E_{\f, m}$ can be regarded as a graphical radial pertubation of the ball $B^{(m)}$. We finally recall that given a sufficiently regular set $E$,  $H_{\partial E}$ stands for the sum of the principal curvatures of $\pa E$ (with respect to the orientation given by the outward normal).

\begin{theorem}[\cite{KM}]\label{th:KM}
There exist $\delta\in (0,\frac12)$ and $C>0$ with the following property: For  any  $f \in C^1(\partial B) \cap H^2(\partial B)$ such that   $\|f\|_{C^1}\le \delta$ and $\bary(E_f)=0$,
  we have
$$
\|f\|_{H^1(\partial B)}\le C \| H_{\partial E_f} - (N-1)) \|_{L^2(\partial B)} \,.
$$
\end{theorem}
\begin{remark}
In fact, in \cite{KM} the theorem is stated under the assumption $f\in C^{1,1}(\pa B)$. However, it is immediate to see by a standard approximation argument that the statement holds true under the weaker hypothesis $f \in C^1(\partial B) \cap H^2(\partial B)$. 
\end{remark}

 The stability inequality for nearly spherical sets with almost constant mean curvature that we are going to use in the next sections is the following.

\begin{theorem}\label{Aleq}
There exist $\delta\in (0,\frac12)$ and $C>0$ with the following property: For  any  $f \in C^1(\partial B) \cap H^2(\partial B)$ such that   $\|f\|_{C^1}\le \delta$,   $|E_f|=\omega_N$ and $\bary(E_f)=0$,
  we have
  \beq\label{e:aleq}
\|f\|_{H^1(\partial B)}\le C \| H_{\partial E_f} - \overline H_{\partial E_f} \|_{L^2(\partial B)} \,,
\eeq
where we have set
$$
\overline H_{\partial E_f} : = \medint_{\partial B}H_{\partial E_f}(x+f(x)x)\, d\H^{N-1}\,.
$$
\end{theorem}

\begin{proof}
First of all we notice that, if we take the constant $C$ in \eqref{e:aleq} to be bigger than $\sqrt{\frac{N}2\omega_{N}}$, then it is enough to consider only the case
$\|H_{\partial E_f} - \overline H_{\partial E_f}\|_2\leq 1$.

We write explicit formulas for the perimeter and its first variations for the sets $E_f$: by using the area formula we get
\beq\label{eq:per}
\P(E_f) = \int_{\partial B} (1+f)^{N-1} (1 + (1+f)^{-2} |\nabla f|^2)^{\frac 12} \, d\H^{N-1}\,.
\eeq
Recall that $H_{\partial E_f} $ is the first variation of the perimeter at $E_f$, that is, 
\beq \label{e:var1-1}
\begin{split}
\delta \P(E_f) [\psi]  := &\frac{d}{dt}\P(E_{f+t\psi}) = \int_{\partial E}  (H_{\partial E_f} \nu_{\partial E_f}\big) \cdot x\, \psi(x) \, d\mathcal{H}^{N-1}(x)\\
&= \int_{\partial B}   (H_{\partial E_f} \nu_{\partial E_f}\big)(p) \cdot x
\;\psi  (1+f)^{N-1}   (1 + (1+f)^{-2} |\nabla f|^2)^{\frac 12}\, d\H^{N-1}\,,
\end{split}
\eeq
where for simplicity of notation we have set $p=(1+f(x))x$. On the other hand, by  differentiating under the integral sign, from \eqref{eq:per} we get
\beq \label{e:var1-2}
\begin{split}
\delta \P(E_f) [\psi]
& =\int_{\partial B}  (N-1) (1+f)^{N-2} (1 + (1+f)^{-2} |\nabla f|^2)^{\frac 12} \psi \, d\H^{N-1}
\\
&\qquad-\int_{\partial B} \frac{(1+f)^{N-4}}{(1 + (1+f)^{-2} |\nabla f|^2)^{\frac 12}} |\nabla f|^2 \psi\, d\H^{N-1}\\
&\qquad+ \int_{\partial B} \frac{(1+f)^{N-3}}{(1 + (1+f)^{-2} |\nabla f|^2)^{\frac 12}}  \nabla f \cdot \nabla \psi\, d\H^{N-1},
\end{split}
\eeq
for all $\psi \in C^1(\partial B)$, where $\nabla f$ and $\nabla \psi$ are the tangent gradient of $f$ and $\psi$ on $\partial B$, respectively.

In the following, with a slight abuse of notation,  by the symbol $O(g)$ we mean  any function $h$ of the form $h(x)=r(x)g(x)$, where $|r(x)|\leq C$ for all $x\in \pa B$, with $C>0$ being a constant depending only on the apriori $C^1$-bound $\|f\|_{C^1}\leq\frac12$. 

Since the normal to $E_f$ at a point $p=(1+f(x))x$ with $x\in \partial B$ is given by
\[
\nu_{{\partial E_f}} (p) = \frac{1}{\sqrt{1+(1+f)^{-2}|\nabla f|^2}}\left(-\frac{\nabla f}{1+f} +  x\right),
\]
one gets
 \beq\label{onegets}
 \nu_{{\partial E_f}} (p) \cdot x = \frac{1}{\sqrt{1+(1+f)^{-2}|\nabla f|^2}}.
 \eeq
Therefore, by \eqref{e:var1-1}, \eqref{onegets}
and a simple Taylor expansion, we have that
\begin{align}\label{e:lin1}
\delta  \P(E_f) [\psi] 
&=  \int_{\partial B } (1+R_1) \,H_{\partial E_f}(p) \Big (1+ (N-1) f + R_2 \Big )   \psi  \, d\H^{N-1},
\end{align}
with $|R_1|= O(|\nabla f|^2)$ and $|R_2|= O(|f|^2)$.
Similarly, using \eqref{e:var1-2} and a Taylor expansion, we get 
\begin{align}\label{e:lin2}
\delta  \P(E_f) [\psi] &= \int_{\partial B} \big(
(N-1) + (N-1)(N-2) f\,+ O(|f|^2) + O(|\nabla f|^2)\big)\psi\, d\H^{N-1}\notag\\
&\quad + \int_{\partial B} \left(\nabla f + h\right)\cdot
\nabla \psi \, d\H^{N-1},
\end{align}
where $h$ is a vector field satisfying 
\beq\label{bacca}
|h|\leq C (|f|+|\nabla f|^2)|\nabla f|\,,
\eeq
with $C>0$ depending only  on the apriori $C^1$-bound $\|f\|_{C^1}\leq \frac12$.
By comparing \eqref{e:lin1} and \eqref{e:lin2}, and recalling that 
$|R_2|= O(|f|^2)$, we infer that
\begin{align}\label{e:lin3}
&\int_{\partial B} \left(-(N-1)f\psi + \nabla f\cdot \nabla \psi\right)\, d\H^{N-1}\notag\\
&= \int_{\partial B} (1+R_1)\big( H_{\partial E_f}(p) - (N-1)\big)(1+(N-1)f+R_2)\psi\, d\H^{N-1}\notag\\
&+\int_{\partial B} \big(-h\cdot \nabla \psi + \big(O(|f|^2) + O(|\nabla f|^2)\big)\psi\big) \, d\H^{N-1}.
\end{align}
Observe that testing \eqref{e:lin3} with $\psi = 1$ and setting 
$R_3= (1+R_1)(1+(N-1)f+R_2)-1$,
we get $R_3 = O(|f|) + O(|\nabla f|^2)$ and 
\begin{equation}\label{e:lin4}
\int_{\partial B} (1+R_3)\big( H_{\partial E_f}(p) - (N-1)\big)\, d\H^{N-1}
=
 \int_{\partial B} \left(O(|f|) + O(|\nabla f|^2)\right)\, d\H^{N-1}\,.
\end{equation}
In particular, for $F :=(1+R_3) H_{\pa E_f}$ and recalling that the overline denotes the average on $\partial B$, we infer that for every $\eta\in (0,\frac12)$,
if $\de$ is sufficiently small, we have that
\[
\left\vert\overline F-(N-1)\medint_{\pa B} (R_3+1)\right\vert \leq \eta
\]
and, always assuming $\delta$ small enough,
\begin{align}
|\overline H_{\partial E_f}-(N-1)|&\leq |\overline H_{\partial E_f}-\overline F| + \left\vert\overline F-(N-1)\medint_{\pa B} (R_3+1)\right\vert+(N-1)\left\vert\medint_{\pa B} R_3\right\vert\notag\\
&\leq\left\vert \medint_{\pa B} R_3  H_{\partial E_f} \right\vert + 2\eta\leq\notag\\
&\leq \omega_N^{-1}\|R_3\|_2 \|H_{\partial E_f} - \overline H_{\partial E_f}\|_2 
+  |\overline H_{\partial E_f}-(N-1)| \int_{\pa B}\left\vert R_3  \right\vert\, d\H^{N-1}+\notag\\
&\qquad + (N-1)\int_{\pa B}\left\vert R_3  \right\vert\, d\H^{N-1}+2\eta\notag\\
&\leq \eta \|H_{\partial E_f} - \overline H_{\partial E_f}\|_2 + \eta |\overline H_{\partial E_f}-(N-1)| + 3\eta\notag\\
&\leq \eta |\overline H_{\partial E_f}-(N-1)| + 4\eta,\notag
\end{align}
where in the last inequality we have used $\|H_{\partial E_f} - \overline H_{\partial E_f}\|_2\leq 1$.
In particuar, there exists a constant $\lambda$ such that
\begin{equation}\label{e:vicinanza media}
\overline H_{\partial E_f}=N-1 + \lambda,\qquad
|\lambda|\leq \frac{4\eta}{1-\eta}.
\end{equation}
Note that $\lambda$ is arbitralily small, if $\delta$ (and hence $\eta$) is chosen accordingly.

Now the proof can be concluded as follows. Consider $\kappa = 1+\frac{\lambda}{N-1}$.
If $\delta$ is small enough, then $\kappa\in (\frac12, 2)$.
Consider also the set  $\kappa E_{f}= E_{u}$ with $u:=\kappa-1+\kappa f$.
Then, 
\[
\overline H_{E_u} = N-1, \qquad \|u\|_{H^1}\leq  \kappa \|f\|_{H^1}+ \sqrt{N\omega_N}\frac{\lambda}{N-1}, \text{ and } \|u\|_{C^1}\leq \kappa \|f\|_{C^1}+ \frac{\lambda}{N-1}\,.
\]
In particular, if $\de$ and  $\eta$ are sufficiently small (and, therefore, such is $\lambda$), we are in position to apply Theorem~\ref{th:KM} (because $\|u\|_{C^1}$ becomes arbitrarily small)
and  infer that
\[
\|u\|_{H^1}\leq C \|H_{E_u} - N+1\|_{L^2} = \kappa^{-1} \,C\|H_{E_f} - \overline H_{E_f}\|_{L^2}
\leq 2C\|H_{E_f} - \overline H_{E_f}\|_{L^2}.
\]
In particular, we can estimate the $L^2$ norm of the gradient of $f$:
\beq\label{eq:gradiente}
\|\nabla f\|_{L^2}\leq \kappa^{-1}\|u\|_{H^1}\leq 4C\|H_{E_f} - \overline H_{E_f}\|_{L^2}.
\eeq
Finally, since $|E_f|=|B|$, i.e.,
\[
\frac{1}{N}\int_{\partial B}(1+f)^N\,d\H^{N-1} = \omega_N,
\]
we get by Taylor expansion
\begin{equation}\label{e:lin6}
\big\vert\int_{\partial B} f\,d\H^{N-1}\big\vert = \int_{\partial B} O(|f|^2) \,d\H^{N-1}.
\end{equation}
This implies that
\begin{equation}\label{e:media}
|\overline f| \leq C\, \|f\|_{L^2}^2\leq C\delta \, \|f\|_{L^2}.
\end{equation}
By Poincar\'e inequality we have that 
\begin{align}\label{e:norma 2}
\|f\|_{L^2} & \leq 2\|f - \overline f\|_{L^2} + 2|\overline f| \big(N\omega_N\big)^{\frac12}\leq
C\|\nabla f\|_{L^2}+ 2|\overline f| \big(N\omega_N\big)^{\frac12}.
\end{align}
Inserting \eqref{e:media} in \eqref{e:norma 2} and combining with \eqref{eq:gradiente} we deduce \eqref{e:aleq} if $\delta$ is sufficiently small.

\end{proof}

\begin{remark}\label{rm:acdc}
By a simple rescaling argument it is clear that
Theorem~\ref{Aleq} holds also for the sets $E_{\f,m}$ parametrized over a ball $B^{(m)}$ with volume $m$, with constants $\delta$ and $C$ depending on $m$.
Clearly, the dependence of such constants on $m$ can be made uniform when $m$ varies in compact subsets of $(0,+\infty)$.
\end{remark}

\begin{remark}
The estimate \eqref{e:aleq} is optimal for what concerns the power of the norms.
To see this, it is enough to consider sets $E_{\varepsilon f}$ with $f$ a functions in the second eigenspace of the Laplace-Beltrami operator on the sphere (with corresponding eigenvalue $2N$) and $\varepsilon>0$ sufficiently small: then, computing the $L^2$-norm of the average of the mean curvature yields
\[
\| H_{\partial E_{\varepsilon f}} - \overline H_{\partial E_{\varepsilon f}} \|_{L^2(\partial B)}\leq 
C  \|\varepsilon f\|_{L^2(\partial B)},
\]
 for a dimensional constant $C<0$.
\end{remark}



\section{The incremental problem}


We start by introducing the incremental minimum problem which defines the discrite-in-time volume preserving mean curvature flow. To this purpose,  let $E\neq \emptyset$ be a bounded measurable set. Notice that the topological boundary of $E$ depends on its representative; from now on, we will assume that $E$ coincides with its Lebesgue representative, i.e., with its points of density equal to one.  We let
\begin{equation}\label{defsigndist}
d_E(x)\ =\ \dist(x,E)-\dist(x,\R^N\setminus E)
\end{equation}
be the signed distance function to $\partial E$. 
Let ${m}>0$ be fixed, representing the volume of the evolving set. 

Fix a time step $h>0$ and consider the problem
\begin{equation}\label{varprob}
\min \left\{P(F)\ +\ \frac{1}{h}\int_F d_E(x)\,dx:\, | F|={m} \right\}.
\end{equation}

Note that 
$$
\int_Fd_E(x)\, dx-\int_E d_E(x)\, dx=\int_{E\Delta F}\mathrm{dist}(x, \partial E)\, dx
$$
so that \eqref{varprob} is equivalent to 
$$
\min \left\{P(F)\ +\ \frac{1}{h}\int_{E\Delta F}\mathrm{dist}(x, \partial E)\, dx
:\, |F|={m} \right\}.
$$
Given two sets $E,\, F$, we let
$$
\D (F,E):=  \int_{E\Delta F}\mathrm{dist}(x, \partial E)\, dx\,.
$$
In order to to prove the existence of a solution to \eqref{varprob}, we need some preliminary results. 
We start with the following non-vanishing estimate for sets of finite perimeter  and finite measure.

\begin{lemma}\label{lm:nonvanishing}
There exists a constant $C=C(N)\in (0, \frac12)$ such that if $E\subset\R^N$ is a set of finite perimeter and finite measure, then, setting $Q:=(0,1)^N$, we have 
$$
\sup_{z\in \Z^N}|E\cap (z+Q)|\geq c(N)\min\Bigl\{\Big(\frac{|E|}{P(E)}\Big)^N, 1\Bigr\}\,.
$$ 
\end{lemma}
\begin{proof}
Settting
$$
\beta:=\sup_{z\in \Z^N}|E\cap (z+Q)|\,,
$$
assume that $\beta\leq \frac12$. Then, 
 by the Relative Isoperimetric Inequality we have
\begin{align*}
P(E)& \geq \sum_{z\in \Z^N}P(E, z+Q)\geq c(N)\sum_{z\in \Z^N}|E\cap (z+Q)|^{\frac{N-1}{N}}\\
&\geq \frac{c(N)}{\beta^{\frac1N}}\sum_{z\in \Z^N}|E\cap (z+Q)|=
\frac{c(N)}{\beta^{\frac1N}} |E|
\end{align*}
and the conclusion easily follows.
\end{proof}

We are now in a position to prove the following proposition, which in particular establishes the existence of a solution to  \eqref{varprob}.
The crucial point in the following statement is that the choice of the penalization parameter $\Lambda$ depends only on the bounds on the perimeter and on the prescribed measure $m$, and thus can be made uniform along the minimizing movements scheme.
\begin{proposition}\label{prop:penal}
Given $m$, $M>0$, there exists $\Lambda_0=\Lambda_0(m, M, h, N)>0$  such that  for any bounded set $E$ of finite perimeter, with $P(E)\leq  M$ and $ |E|=m$, and for any 
$\Lambda\geq\Lambda_0$, any solution  solution $\overline F$ of
\begin{equation}\label{varprob-pen}
\min \left\{P(F)\ +\frac1h\D(F, E)+\Lambda\big||F|-m\big|:\, F\subset\R^N\textrm{ \em  measurable}\right\}
\end{equation}
satisfies the volume constraint $|\overline F|=m$. In particular, \eqref{varprob} and \eqref{varprob-pen} are equivalent. 
\end{proposition}
\begin{proof}
First of all, note that the existence of a solution to \eqref{varprob-pen} follows by standard arguments, see for instance \cite{MSS}.
We argue by contradiction by assuming for every $n$ the existence of a set $E_n$, with $|E_n|=m$ and 
$P(E_n)\leq M$,  and a set 
$$
\overline F_n\in \operatorname*{argmin}\left\{P(F)\ +\frac1h\D(F, E_n)+n\big||F|-m\big|:\, F\subset\R^N\text{ measurable} \right\}
$$
such that $|\ov F_n|\neq m$. In the sequel, we may assume $|\ov F_n|<m$, as the other case can be treated analogously. Testing with $E_n$, by minimality we have
\beq\label{trbounds}
P(\ov F_n)\ +\frac1h\D(\ov F_n, E_n)+n\big||\ov F_n|-m\big|\leq P(E_n)\leq M\,.
\eeq
In particular, 
\beq\label{plus}
|\ov F_n|\to m  
\eeq
as $n\to\infty$. In turn, by Lemma\til\ref{lm:nonvanishing} there exist a constant $c_0>0$, depending only on  $m$, $M$ and $N$, and $z_n\in \Z^N$ such that 
$$
|\ov F_n\cap (z_n+Q)|\geq c_0
$$
for all $n$ sufficiently large. Thus, by replacing $\ov F_n$, $E_n$ by $\ov F_n-z_n$, $E_n-z_n$, respectively,   and appealing to  the well-known campactness properties of sets of finite perimeter, we may assume that up to a (not relabelled) subsequence we have $\ov F_n\to F_\infty$ in $L^1_{loc}$, with $|F_\infty|\geq|F_\infty\cap Q|\geq c_0>0$. 

The idea now is to modify the sets $\ov F_n$ by applying a local dilation in order to impose the volume constraint. In this construction, we follow closely  \cite[Section\til2]{EF}.  We only outline the main steps. 

First of all, arguing as in  Step 1 of  \cite[Section\til2]{EF},  for  any fixed $y_0\in \partial^*F_\infty$ and  for any given small $\e>0$  we may find a radius  $r>0$ and a point $x_0$ in a small neighborhood of $y_0$   such that  
 $$
 | \ov F_n\cap B_{r/2}(x_0)|<\e r^N\,, \quad | \ov F_n\cap B_{r}(x_0)|>\frac{\omega_Nr^N}{2^{N+2}}
 $$
 for all $n$ sufficiently large. In the following, to simplify the notation we assume that $x_0=0$ and we write $B_r$ instead of $B_r(0)$. For a sequence $0<\sigma_n<1/2^N$ to be chosen, we introduce the following bilipschitz maps:
  $$
  \Phi_n(x):=
 \begin{cases}
 (1-\sigma_n(2^N-1))x & \text{if $|x|\leq \frac r2$,}\\
 x+\sigma_n\Bigl(1-\frac{r^N}{|x|^N}\Bigr)x & \text{$\frac r2\leq |x|<r$,}\\
 x & \text{$|x|\geq r$.}
 \end{cases}
 $$
 Setting $\widetilde F_n:=\Phi_n(\ov F_n)$, we have as in Step 3 of \cite[Section\til2]{EF}
 \beq\label{EFper}
 P(\ov F_n, {B_r})-P(\widetilde F_n, {B_r})\geq -2^NNP (\ov F_n, {B_r})\sigma_n\geq 
 -2^NN M\sigma_n\,.
 \eeq
 Moreover, as in Step 4  of  \cite[Section\til2]{EF} we have
$$
 |\widetilde F_n|-|\ov F_n|\geq \sigma_nr^N\Bigl[c\frac{\omega_N}{2^{N+2}}-\e(c+(2^N-1)N)\Bigr]
$$
 for a suitable constant $c$ depending only on the dimension $N$.
If we fix $\e$ so that the negative term  in the square bracket does not exceed half  the  positive one, then we have 
  \beq\label{EFvol}
   |\widetilde F_n|-|\ov F_n|\geq \sigma_nr^NC_1\,,
  \eeq
  with $C_1>0$ depending on $N$.
  In particular, from this inequality it is clear that we can choose $\sigma_n$ so that $|\widetilde F_n|=m$ for $n$ large; this implies  $\sigma_n\to 0$, thanks to \eqref{plus}.
  
Now, arguing as in \cite[Equations (2.12) and (2.13)]{AFM}, we obtain
\beq\label{difsim}
  |\widetilde F_n\Delta \ov F_n|\leq C_3\sigma_nP(\ov F_n, B_r)\leq C_3\sigma_n M\,.
\eeq
Set now 
$$
i_n:=\min_{\ov B_r}\dist(\cdot, \partial E_n)
$$
and note that if $i_n>0$, then either $B_r\subset E_n$ or $B_r\subset E_n^c$. In the first case, we have
$(E_n\Delta \ov F_n)\cap B_r=B_r\setminus \ov F_n$ and, in turn, 
$$
|B_r\setminus \ov F_n|\geq |B_{r/2}\setminus \ov F_n|=|B_{r/2}|-| \ov F_n\cap B_{r/2}|>\Big(\frac{\omega_N}{2^N}-\e\Big) r^N\geq \frac{\omega_N}{4^N}r^N\,,
$$
by choosing $\e$ smaller if needed.  
In turn, recalling \eqref{trbounds}, we may estimate
$$
M\geq \frac1h\D(\ov F_n, E_n)\geq \frac1h \int_{B_r\setminus \ov F_n}\dist(x, \partial E_n)\, dx\geq
\frac{i_n}h\frac{\omega_N}{4^N}r^N\,,
$$
from which we easily deduce 
$$
\dist(\cdot, \partial E_n)\leq 4^{N}r^{-N}M \omega_N^{-1}h+2r \quad\text{on }B_r\,. 
$$
Arguing in a similar way also in the case $B_r\subset E_n^c$, we can finally conclude the existence
 of a positive constant $C_4=C_4(r)$ (depending also on $N$ and $h$) such that, in all cases, and for $n$ large enough we have
 $\dist(\cdot, \partial E_n)\leq h C_4$  on $B_r\,$ and thus 
 \beq\label{Vn1000}
\Big|\frac1h\D(\widetilde F_n, E_n)-\frac1h\D(\ov F_n, E_n)\Big|\leq C_4  |\widetilde F_n\Delta \ov F_n|\leq 
C_4 C_3\sigma_n M\,,
\eeq
where in last inequality we have used \eqref{difsim}.
 Combining \eqref{EFper}, \eqref{EFvol},  and \eqref{Vn1000},  we conclude that for $n$ sufficiently large
 \begin{align*}
 P(\widetilde F_n)+\frac1h\D(\widetilde F_n, E_n)\leq &  P(\ov F_n)+\frac1h\D(\ov F_n, E_n)+n\big||\ov F_n|-m\big|
 \\ &+
 \sigma_n\bigl[(2^NN+ C_4C_3)M- n r^NC_1\bigr]
 \\
 <& P(\ov F_n)+\frac1h\D(\ov F_n, E_n)+n\big||\ov F_n|-m\big| \,,
\end{align*}
a contradiction to the minimality of $\ov F_n$, since  $|\widetilde F_n|=m$. 
 \end{proof}

As a consequence of the previous proposition, together with some standard arguments from the regularity theory of  almost minimal sets, we have the following:

\begin{proposition}[Regularity properties of minimizers]\label{lm:density}
	Let $E$ be a bounded set with $|E|=m$ and $P(E)\leq M$ for some $m,\, M>0$. 
	Then, any solution $F\subset \R^N$ to \eqref{varprob} satisfies the following regularity properties:
	\begin{itemize}	
	\item[i)] There exist  $c_0=c_0(N)>0$ and   a radius $r_0=r_0(m, M, h, N)>0$ such that  
		for every $x\in \partial^{*} F$ and $r\in (0, r_0]$ we have
	\beq\label{eq:density}
	|B_r(x)\cap F|\geq c_0r^N\qquad\text{and}\qquad |B_r(x)\setminus F|\geq c_0r^N\,.
	\eeq
In particular, $F$ admits an open representative  whose topological boundary coincides with the closure of its reduced boundary, i.e., $\partial F = \overline{ \pa^* F}$.         
From now on we will always assume that $F$ coincides with its open representative.
	\item [ii)]
         There exists
          $c_1= c_1(m,M,h,N)>0$ such that   
	\beq\label{neigh}
	\sup_{E\Delta F} \dist(\cdot, \partial E)\leq c_1 \, . 
	\eeq

	\item[iii)] There exists $\overline\Lambda=\overline\Lambda(m, M, h, N)>0$ such that $F$ is a 
	 $\overline\Lambda$-minimizer of the perimeter, that is, 
	 \beq\label{almost}
	 P(F)\leq P(F')+\overline\Lambda |F\Delta F'|
	 \eeq
	 for all measurable set $F'\subset \R^N$  such that diam$(F\Delta F') \le 1$. 
	\item[iv)] The following Euler-Lagrange condition holds: There exists $\lambda\in \R$ such that for all $X\in C^{1}_c(\R^N; \R^N)$
\begin{equation}\label{micume0}
\int_{\partial^* F} \frac{d_E(x)}h X\cdot\nu_F\, d\H^{N-1} +\int_{\partial^* F}\div_\tau X\, d\H^{N-1}=\lambda\int_{\partial^* F}X\cdot\nu_F\, d\H^{N-1}\,.
\end{equation}
\item[v)] There exists a closed set $\Sigma$ whose Hausdorff dimension is less than or equal to $N-8$, such that $\partial^* F= \partial F\setminus \Sigma$  is an $(N-1)$-submanifold of class $C^{2,\alpha}$ for all $\alpha\in (0,1)$ with 
	$$
	|H_{\partial F} (x)| \le \overline \Lambda \quad \text{ for all } x\in \partial F\setminus \Sigma \, .
	$$
	
\item  [vi)] The set $F$ is bounded and more precisely there exist $k_0=k_0(m, M, h, N)\in \N$ and $d_0=d_0(m, M, h, N)>0$ such that $F$ is made up of at most $k_0$ connected components,
each one having diameter bounded  from above by $d_0$
(a bound on the diameter from below follows from \textup{i)}).
\end{itemize}
	\end{proposition}


\begin{proof}
By Proposition~\ref{prop:penal}, there exists $\Lambda_0=\Lambda_0(m, M, h, N)$ such that  $F$ is a solution to 
\beq\label{eq:penalbis}
\left\{P(F)\ +\frac1h\D(F, E)+\Lambda_0\big||F|-m\big|:\, F\subset\R^N\text{ measurable} \right\}\,.
\eeq
The density estimates and \eqref{neigh} follow arguing as in \cite[ Lemma~4.4 and Proposition~3.2]{MSS}, respectively.
Now, \eqref{almost} easily follows from the fact that $F$ solves \eqref{eq:penalbis}, taking into account \eqref{neigh}.

Equation \eqref{micume0} can be derived by a standard first variation argument (see for instance \cite{MaggiBook}).  In view of (iii) and the classical  regularity theory for almost minimizers of the perimeter (see for instance \cite{MaggiBook} and the references therein), we deduce that 
there exists a closed set $\Sigma$ whose Hausdorff dimension is less than or equal to $N-8$, such that $\partial^* F= \partial F\setminus \Sigma$  is an $(N-1)$-submanifold of class $C^{1,\alpha}$ for all $\alpha\in (0,1)$. The $C^{2,\alpha}$ regularity stated in (v)  follows now from  the additional elliptic regularity implied by \eqref{micume0},   taking into account that $d_E$ is Lipschitz continuous. 
 
Item (vi) follows rather easily from the density estimates.  Indeed, let $r_0=r_0(m, M, h, N)>0$ be the radius given 
  in i). By an application of Vitali's Covering Lemma, we may find a subset $C\subset F$ such that the closed balls of the family $\{\overline{B_{r_0}(x)}\}_{x\in C}$  are pairwise disjoint and 
 $$
 F\subset \bigcup_{x\in C} \overline{B_{5r_0}(x)}\,.
 $$
 Since by \eqref{eq:density} we have $|B_{r_0}(x)\cap F|\geq c_0r_0^N$ for every $x\in C$, it follows that $\# C\leq \frac{|F|}{c_0r_0^N}=
 \frac{m}{c_0r_0^N}$. In turn, since for each connected component $\hat F$ of $F$ we have $\hat F\subset \cup_{x\in C} \overline{B_{5r_0}(x)}$,  we infer that
 $$\textrm{diam\,}\hat F\leq 10 r_0 \# C\leq  \frac{10 m}{c_0r_0^{N-1}}=:d_0(m, M, h, N)\,.$$ 
 

 It remains to bound the number of connected components. To this aim,  it is enough to show that  the measure of any connected component $\hat F$ is bounded below by a positive constant depending only on $m$, $M$, $h$,  and $N$.  
 Set $F':=F\setminus \hat F$.
Then, using  \eqref{almost}  and  the Isoperimetric Inequality, we deduce

 \begin{equation*}
 \ov\Lambda \big|\hat F\big|= \ov\Lambda \big| F\Delta F'\big|+ P(F')-(P(F)-P(\hat F))\geq P(\hat F)\geq N\omega_N^{\frac{1}{N}}\big|\hat F\big|^{\frac{N-1}{N}}\,,
 \end{equation*}
and thus  
 $$
\big|\hat F\big|\geq\Big(\frac{N}{\ov\Lambda}\Big)^N\omega_N\,.
 $$
 This concludes the proof of the proposition.
\end{proof}


%

\begin{remark} Note that \eqref{micume0} and the regularity properties given by Proposition~\ref{lm:density} imply that 
\begin{equation}\label{micume}
  \frac{d_E(x)}h+H_{\partial F}(x)=\lambda \qquad \text{ for every } x\in\partial^* F.
\end{equation}
\end{remark}

\section{The discrete volume preserving flow}

\subsection{Construction of the discrete flow}
 For any bounded  set $E\neq \emptyset$ with finite perimeter we let  $T_h E$  denote  a solution of \eqref{varprob}, with $m=|E|$.  
   It is convenient to fix a precise representative for $T_h E$;
as done for the set $E$, we assume that $T_h E$ coincides with its  Lebesgue representative.
Now, we construct by induction the discrete-in-time evolution $\{E_h^n\}_{n\in N}$. Let  $E_h^1:=T_h  E$ be a solution (arbitrarily chosen) to problem \eqref{varprob}; assuming that $E_h^k$ is defined for $k\in \{1,\ldots, n-1\}$, we let  
$E_h^{n}$ be a solution (arbitrarily chosen) to problem \eqref{varprob} with $E$ replaced by $E_h^{n-1} $.

\subsection{Stationary sets for the discrete flow}
In this section we characterize the stationary sets $E$ for the  volume preserving discrete flow, i.e., 
such that the constant sequence $E_h^n\equiv E$ 
is a  discrete volume preserving flow.

\begin{proposition}[Fixed points of the discrete scheme]\label{fpds}

Given $m, \, M, \, h>0$, there exists $s_0= s_0(m,M,h,N)>0$ such that every 
	stationary bounded set $E$ for the volume preserving discrete flow with time step $h$, with $|E|=m$ and $P(E)\leq M$  
	is made up by the union of $k$ disjoint balls with mutual distances larger that $s_0$
	and  equal volume $\frac mk$, for some $k \le k_0$, with $k_0$ as in item \textup{vi)} of Proposition \ref{lm:density}.


Viceversa, if $E$ is given  by the union of finitely many disjoint balls with positive mutual distances and  equal volumes, then there exists $h^*>0$ such that, for all $h\le h^*$, 
the  volume preserving flow $\{E_h^n\}$ starting from $E$ is unique and given by the constant sequence $E_h^n= E$ for all $n\in \N$.
\end{proposition}
\begin{proof}
By \eqref{micume0} any stationary set $E$ satisfies
$$
\int_{\partial^*E}\div_\tau X\, d\H^{N-1}=\lambda\int_{\partial^*E}X\cdot\nu_E\, d\H^{N-1}
$$ 
 for all $X\in C^{1}_c(\R^N; \R^N)$  and for some $\lambda\in \R$.
  We may then apply \cite[Theorem~1]{DM} to conclude that $E$ is given (up to a negligible set) by a finite union of disjoint (open) balls of equal volume.  Moreover, by (vi) of Proposition \ref{lm:density}
  the number of such balls is bounded by $k_0$, and therefore,  in view of \eqref{eq:density}, their radius  is bounded from below by a constant $R$ depending only on $m,M,h,N$. 
 Now, if two balls of radius larger than or equal to $R$ are at a distance $s$ small enough, then $E$ cannot satisfy the second density estimate in \eqref{eq:density}  (for example, \eqref{eq:density} cannot hold for $x\in\partial E$ of minimal distance between the two balls and $r>0$ depending on $R$ and the constants $r_0$, $c_0$ in Proposition \ref{lm:density}-i)). Therefore,  \eqref{eq:density} is violated, whenever $s \le s_0$ for a suitable  $s_0$ depending only on $m,M,h,N$, 
 thus establishing the first part of the statement. 
 
 Assume now that $E$ is union of finitely many disjoint balls with equal radius $R$ and positive mutual distances. We want to show that, for $h$ small enough, $E$ is the unique volume costrained global minimizer
 of the functional 
 $$J_h(F):= P(F)\ +\frac1h\D(F, E) =  P(F) + \frac1h\int_Fd_E(x)\, dx- \frac1h\int_E d_E(x).
 $$ 
 
 We start by showing that for $h$ small enough the  second variation of $J_h$ is positive definite with respect to volume preserving variations. 
  To this purpose, let $X\in C^1_c(\R^N;\R^N)$ be a divergence free vector field and let $\Phi(t,\cdot)$ be the associated flow satisfying $\frac{\partial \Phi}{\partial t} = X(\Phi)$ with initial condition $\Phi(0,x)= x$ for all $x\in\R^N$. Then $|\Phi(t,E)|=|E|= m$ for all $t$, and  by standard computations (see for instance \cite{AFM, Cr}) we have
 \begin{align*}
  \frac{\partial^2 }{\partial t^2}  J_h(\Phi(t,E) ) 
&= \int_{\partial E} |\nabla (X \cdot \nu_E)|^2 - \frac{N-1}{R^2} (X \cdot \nu_E)^2 \, d\H^{N-1} 
 \\
\nonumber & \quad + \frac 1h \int_{\partial E} \partial_{\nu_E} d_E (X\cdot \nu_E)^2 \, d\H^{N-1}  
\\
\nonumber
&= \int_{\partial E} |\nabla (X \cdot \nu_E)|^2 
+ \Big( \frac 1h - \frac{N-1}{R^2} \Big) 
(X \cdot \nu_E)^2 \, d\H^{N-1}
\\
\nonumber
&= : \partial^2 J_h(E)[X\cdot\nu_E], 
 \end{align*}
 where in second  equality we have used the fact that $\partial_{\nu_E} d_E \equiv 1$ on $\partial E$.   From the above expression it is clear that $\partial^2 J_h(E)$ is positive definite on $H^1(\pa E)$, provided that
 $h <  \frac{R^2}{N-1}$.
 
Fix   $h_0 <  \frac{R^2}{N-1}$. Arguing as in \cite{AFM} there exists $\e>0$ such that $J_{h_0} (E) < J_{h_0} (F)$ for all measurable $F$ such that $|F|=|E|$ and $|E\Delta F|\le \e$. 
Now notice that for all $0<h<h_0$ we have 
\begin{equation}\label{lomi}
J_h(E) = J_{h_0}(E)  < J_{h_0} (F) \le  J_{h} (F) \qquad \text{ for all $|F|=|E|$ and $0< |E\Delta F|\le \e$},
\end{equation} 
i.e., $E$ is an isolated local minimizer for $J_h$ in $L^1$, with minimality neighborhood uniform with respect to $h\le h_0$.
 	Now, given any sequence $\{h_n\}$ going to zero, let $F_n$ be a volume constrained minimizer of $J_{h_n}$; it is easy to see that $|E\Delta F_n|\to 0$ as $n\to +\infty$, and therefore, for $n$ large enough, $|E\Delta F_n| \le \e$, so that by 
\eqref{lomi} $F_n=E$.

\end{proof}
\begin{remark}
Let us also observe that  for $N\leq 7$ we have full regularity of $\partial  E$ and thus, in particular, the connected components of $E$ have positive mutual distances  and the conclusion of the first part of  Proposition \ref{fpds} would also follow by applying the classical Alexandrov's Theorem instead of the more refined results of \cite{DM}.
\end{remark}

\subsection{Long-time behaviour}

In the following, we denote by $P_\infty$ the limit of the monotone non-increasing sequence $\{P(E_h^{n})\}_{n\in\N}$:
\begin{equation}\label{e:Pinfinito}
 P_\infty = \lim_{n\to \infty}P(E_h^{n}).
\end{equation}

The following is the main result of the paper on the long time behavior of the discrete-in-time nonlocal mean curvature flow.

\begin{theorem}\label{mainthm}
Let $E$ be a bounded set of finite perimeter with $|E|=m$ and let $h>0$.
 Consider any discrete volume constrained mean curvature flow $\{E^{n}_h\}_{n\in\N}$ starting from $E$. 
Then,  setting $L:= N^{-N} \omega_N m^{1-N} P_\infty^{N} \in \N$,  where $P_\infty$ is given in \eqref{e:Pinfinito}, there exist $L$ distinct balls 
$B^1, \ldots, B^L$ with the same radius and at positive distance to each other, 
such that
the sets $E^{n}_h$  converge to $E_\infty:= \bigcup_{i=1}^L B^i$ 
in $C^k$ for every $k\in\N$.  
Moreover, the convergence is exponentially fast. 
\end{theorem}

 The following is an immediate corollary.

\begin{corollary}\label{cor:menodidue}
If $P(E)< 2 P(\bar B)$, where $\bar B$ is a ball with volume ${\frac{m}{2}}$, or if $E$ is union of two tangent balls, than $E_\infty$ is a ball.
\end{corollary}

\begin{proof}
If $P(E)< 2 P(\bar B)$, than by definition $P_\infty<2 P(\bar B)$ and $L<2$.
On the other hand, two tangent balls are not stationary for the discrete flow, as shown in Proposition \ref{fpds}. This implies that $P(E_h^n)<P(E)= 2P(\bar B)$, so that one can argue as before.
\end{proof}

The rest of this section is devoted to the proof of Theorem \ref{mainthm}.
 We start with the following lemma.


%
%
%
%
%
%
%
%

\begin{lemma}\label{uniconvd}
Let $\{E_h^n\}_{n\in\N}$ be a volume preserving discrete flow starting from $E$ and let $E_h^{k_n}$ be a subsequence such that $\chi_{E_h^{k_n} -\tau_n}\to \chi_F$ in $L^1(\R^N)$ for some set $F$ and a suitable sequence $(\tau_n)_{n\in\N}\subset \R^N$. 
Then $d_{E_h^{k_n-1} } (\cdot + \tau_n) \to d_F$ locally uniformly in $\R^N$.
\end{lemma}
 \begin{proof}
 We start by observing that for very $n$, by the minimality  of $E_h^n$, using $E_h^{n-1}$ as a competitor,  we have
\beq\label{triviale}
 \frac 1h \D(E_h^n,E_{h}^{n-1}) \le \P(E_h^{n-1}) - \P(E_h^n).
\eeq
Therefore, summing over $n$ we get 
$$
\sum_{n=1}^{\infty}  \frac 1h \D(E_h^n,E_{h}^{n-1}) \le  \P(E)\,.
$$
In particular,
\beq\label{dissto0}
\D(E_h^n,E_{h}^{n-1}) \to 0 \qquad\text{as }n\to\infty\,.
\eeq
Passing to a further (not relabelled) subsequence we may assume that there exists $G$ such that  $E_{h}^{k_n-1}-\tau_n\to G$ in $L^1_{loc}(\R^N)$. In fact, by the density estimates  \eqref{eq:density} and standard arguments the convergence holds in the Kuratowski sense together with the   Kuratowski  convergence of their boundaries. Thus, 
$$
\dist(\cdot, \pa E_{h}^{k_n-1}-\tau_n)\to \dist(\cdot, \pa G) \qquad\text{locally uniformly in }\R^N\,.
$$
Combining the latter information with \eqref{dissto0}, for every $R>0$  we easily get
$$
0=\lim_{n\to\infty}\int_{(E_h^{k_n}\Delta E_{h}^{k_n-1}-\tau_n)\cap B_R}\dist(x, \pa E_{h}^{k_n-1}-\tau_n)\, d\H^{N-1}=\int_{(F\Delta G)\cap B_R}\dist(x, \pa G)\, d\H^{N-1}.
$$
 The last equality yields $F\Delta G\subset \pa G$. As $G$ still satisfies the density estimates we have $|\pa G|=0$ and the conclusion follows. 
 \end{proof}
 
 In the next proposition we start by showing that for any discrete volume preserving flow  the evolving sets  are eventually made up of a constant  number $L$ of connected components, each of which converging up to translations to  the same ball   of volume $m/L$. After establishing such a result, it will remain to show that the magnitude of such translations decays to zero exponentially fast.
\begin{proposition}[Long-time behavior up to translations]\label{prop:uptrans}
Let $\{E_h^n\}_{n\in\N}$ be a discrete flat flow with time step $h>0$ and prescribed volume $m>0$.
Let $P_\infty$ and $L$ be as in the statement of Theorem \ref{mainthm}.
Then, for $n$ sufficiently large, $E_h^n$ has $L$ distinct connected components $E_{h,1}^n, \ldots, E_{h,L}^n$, such that $\textup{dist}(E_{h,i}^n, E_{h,j}^n)\geq \frac{s_0}{2}$ for $i\neq j$ (with $s_0$ the constant of Proposition \ref{fpds}), and  $E_{h,i}^n-\bary (E_{h,i}^n)$ converge to the ball centered at the origin with volume $\frac{m}{L}$ in $C^k$ for every $k\in \N$.
\end{proposition}

\begin{proof}
It sufficies to prove that, given any subsequence $\{E_{h}^{k_n}\}_{n\in\N}$, there existe a sub-subsequence (not relabelled) satisfying the conclusions of the proposition.

To this purpose, let $\{E_{h}^{k_n}\}_{n\in\N}$ be a given subsequence.
By Propositon \ref{lm:density}-vi), each set $E_h^{k_n}$ is made up of  $l_n\leq k_0$ connected components with uniformly bounded diameter less or equal $d_0$.
Therefore, there exist $l_n$ balls $\{B_{d_0}(\xi_i^n)\}_{i=1, \ldots, l_n}$ (not necessarily disjoint one from the other), each containing a different component of $E_{h}^{k_n}$ and such that $E_h^{n} \subset \cup_{i=1}^{l_n} B_{d_0} (\xi_i^n)$.
Up to passing to a subsequence (not relabelled), we can assume that $l_n= \tilde l$, and for all $1\le i<j\le \tilde l$ the following limits exist
$$
\lim_n |\xi_i^n - \xi_j^n| =: d_{i,j}.
$$
Now we define the following equivalence classes: $i\equiv j$ if and only if $d_{i,j} < + \infty$. Denote by $l$ the number of such equivalent classes, let $j(i)$ be a representative for each class $i\in \{1, \ldots, l\}$,  and set $\sigma_i^n:=\xi_{j(i)}$ for $i=1,\ldots, l$.
We have constructed a subsequence $E_h^{k_n}$ satisfying   
$E_h^{k_n} \subset \cup_{i=1}^l B_R(\sigma_i^n)$, where $R= d_0 + \max\{d_{i,j} : d_{i,j}<\infty\}+1$, and for all $i\neq j$ there holds $|\sigma_i^n-\sigma_j^n|\to +\infty$ as $n\to +\infty$. 
We remark that, while $E_h^{k_n -1}$ are in general not   elements of the subsequence $\{E_h^{k_n}\}_{n\in\N}$, they will still play a role in our arguments, since they will  be involved in exploiting the minimality properties of  $E_h^{k_n}$.

Now, fix $1\le i \le l$, and set 
$$
F^n_i:= (E_h^{k_n} - \sigma_i^n), \qquad  \tilde F^n_i:= (E_h^{k_n} - \sigma_i^n)\cap B_R, \qquad m^n_i:= |\tilde F^n_i|.
$$
Up to a subsequence, we have $m^n_i\to m_i$ for some $m_i>0$. 
Moreover, by Lemma \ref{uniconvd} and the compactness of sets of equi-bounded perimeters, there exist measurable sets $\tilde F_i\subset\subset B_R$ such that (again up to a subsequence) 
\begin{equation}\label{locconv}
\tilde F^n_i \to \tilde F_i \text{ in } L_{}^1, \qquad   d_{E_h^{k_n - 1}}(\cdot + \sigma^n_i) \to d_{\tilde F_i}(\cdot) \text{ locally uniformly}. 
\end{equation}
We shall show that $\tilde F_i$ is stationary for the discrete volume preserving flow.
To this aim, let 
$\tilde G_i$ be any bounded set with $|\tilde G_i|=m_i$.
We define  the homotetically rescaled sets $\tilde G^n_i :=\left(\frac{m_i^n}{m_i}\right)^{\frac1N} \tilde G_i$ such that  $|\tilde G_i^n|=m_i^n$. Note that
$\tilde G^n_i \to \tilde G_i$ in $L^1$ and $P(\tilde G^n_i) \to P(\tilde G_i)$ as $n\to +\infty$. 
We set now $G_i^n:=  F^n_i \cup \tilde G^n_i \setminus \tilde F^n_i$: notice that, since $\tilde G_i^n$ is bounded and  the connected components of  $F_i^n\setminus \tilde F_i^n$ diverge, it follows that, for sufficiently large $n$, $G_i^n$ is made up by the same connected componets of $F_i^n$ except $\tilde F^n_i$ which is replaced by $\tilde G^n_i$; in particular, $|F_i^n| = |G_i^n|$.
By the minimality of $E_h^{k_n}$ we have 
\begin{equation*}
P(F^n_i)\ +\ \frac{1}{h}\int_{F^n_i} d_{E_h^{k_n - 1}}(x + \sigma^n_i)\,dx
\le
P(G^n_i)\ +\ \frac{1}{h}\int_{G^n_i} d_{E_h^{k_n - 1}}(x + \sigma^n_i)\,dx.
\end{equation*}
Since  
	the connected components
	of $F_i^n\setminus \tilde F_i^n$ diverge
 and the two sets $F_i^n$ and $G_i^n$ differ only for the components $\tilde F_i^n$ and $\tilde G_i^n$, by addiditivity the previous inequality is equivalent to 
\begin{equation}
P(\tilde F^n_i)\ +\ \frac{1}{h}\int_{\tilde F^n_i} d_{E_h^{k_n - 1}}(x + \sigma^n_i)\,dx
\le
P(\tilde G^n_i)\ +\ \frac{1}{h}\int_{\tilde G^n_i} d_{E_h^{k_n - 1}}(x + \sigma^n_i)\,dx.
\end{equation}
Passing to the limit as $n\to +\infty$, using \eqref{locconv} and the lower semicontinuity of the perimeter, since the sets $\tilde F_i^n$ and $\tilde G_i^n$ are uniformly bounded, we deduce that
\begin{equation}
P(\tilde F_i)\ +\ \frac{1}{h}\int_{\tilde F_i} d_{\tilde F_i} (x)\,dx
\le
P(\tilde G_i)\ +\ \frac{1}{h}\int_{\tilde G_i} d_{\tilde F_i} (x) \,dx.
\end{equation}
This minimality property extends by density to all competitors $G_i$ with finite perimeter and the same volume $m_i$, so that we deduce that $\tilde F_i$ is a fixed point for the discrete scheme with prescribed volume $m_i$, and whence 
by Lemma \ref{fpds} $\tilde F_i$ is given by the union of disjoint balls with positive mutual distances and equal volume.
Moreover, since $\tilde F_i$ are uniform $\overline{\Lambda}$-minimal by Proposition \ref{lm:density}, from the classical regularity theory (see \cite{MaggiBook}) we also deduce that $\tilde F_i^n$ converge to $\tilde F_i$ in $C^{1, \alpha}$ for every $\alpha\in (0,1)$. In particular, for $n$ large enough, $\tilde F_i^n$ has the same number of connected components of $\tilde F_i$.

Summarizing, we have shown that, for a subsequence (not relabelled) of $E_h^{k_n}$ and for $n$ large enough, $E_h^{k_n}$ is made up by a fixed number $K$ of connected components $E_{h,1}^{k_n},\ldots, E_{h,K}^{k_n}$, each converging to a ball (possibly with different radius, if the components belong to different equivalent classes according to the relation introduced above).
Therefore, for  $i\leq K$ we have  $E_{h,i}^{k_n}- \bary({E_{h,i}^{k_n}} )\to B_{R_i}$, where $B_{R_i}$ is the ball centered at the origin with radius $R_i>0$.

It remains to show that all the radii $R_i$ are equal to $R$: from this and the $C^1$ convergence of the translated components to $B_R$, it follows that $K P(B_{R}) = P_\infty$, i.e., $K=L$.
To this aim, we consider the Euler-Lagrange equation \eqref{micume0}
for $E_{h}^n$: for every $X\in C_c^1(\R^n;\R^n)$,
\[
\int_{\partial E_{h}^{k_n}} \frac{d_{E_{h}^{{k_{n}-1}}}(x)}h X\cdot\nu_{E_{h}^{k_n}}\, d\H^{N-1} +\int_{\partial E_{h}^{k_n}}\div_\tau X\, d\H^{N-1}=\lambda_n\int_{\partial E_{h}^{k_n}}X\cdot\nu_{E_{h}^{k_n}}\, d\H^{N-1}\,.
\]
By Proposition \ref{lm:density}-ii) and v), we deduce that 
\[
|\lambda_n|\leq h^{-1} \|d_{E_{h}^{{k_n}-1}}\|_{L^\infty(E_{h}^{{k_n}})} + \|H_{E_{h}^{{k_n}}}\|_{L^\infty(E_{h}^{{k_n}})}\leq h^{-1} c_1 + |\bar{\Lambda}|.
\]
Therefore, by passing to a further subsequence (not relabelled), we can assume that $\lambda_n\to \lambda$, for some $\lambda\in \R$.
Arguing as before, we can localize the Euler-Lagrange equation to each single $F_i^n$:
\[
\int_{\partial F_{i}^n} \frac{d_{E_{h}^{{k_n}-1}}(x+\sigma_i^n)}h X\cdot\nu_{F_{i}^n}\, d\H^{N-1} +\int_{\partial F_{i}^n}\div_\tau X\, d\H^{N-1}=\lambda_n\int_{\partial F_{i}^n}X\cdot\nu_{F_{i}^n}\, d\H^{N-1}\,.
\]
Passing into the limit as $n\to \infty$ and taking into account Lemma \ref{uniconvd}, we deduce that 
\[
\int_{\partial \tilde F_{i}}\div_\tau X\, d\H^{N-1}=\lambda\int_{\partial \tilde F_{i}}X\cdot\nu_{\tilde F_{i}}\, d\H^{N-1}\,.
\]
In particular, this shows that $R_i=\frac{N-1}{\lambda}$.

Finally, the $C^k$ convergence follows by a classical bootstrap method. The  idea is to describe the boundaries $\pa F_n$ (up to changing the reference frame)  locally as graphs of suitable functions $f_n: B'\to \R$, with $B'$ a ball of $\R^{N-1}$. Then we have
$$
\div\left(\frac{\nabla f_n}{\sqrt{1+|\nabla f_n|^2}}\right)=H_n\,,
$$
with $H_n$ and $\nabla f_n$ uniformly bounded in $C^{0,\alpha}$ for every $\alpha\in (0,1)$. Differentiating the equation in any direction $v$, we get
$$
\textup{div}\big(\nabla A(\nabla f_n) \nabla (\partial_v f_n)\big) =\pa_v H_n\,, 
$$
where the coefficients $A(\nabla f_n) $ are uniformly elliptic and uniformly bounded in the $C^{0,\alpha}$ norm. 
We may therefore apply \cite[Theorem 5.18]{GM} to conclude that  the functions $\partial_v f_n$ are uniformly bounded in the $C^{1,\alpha}$ norm. In order to bound the higher order norms we may now apply a standard bootstrap argument and  iteratively use \cite[Theorem 5.18]{GM}. We leave the details to the reader.
\end{proof}

We recall the notation introduced before \eqref{schifo}: $B^{(\mu)}$ stands for the ball of volume $\mu$ and $r(\mu)$ denotes its radius.

\begin{lemma}\label{disti0}
Let $\mu>0$  and $\eta>0$. There exists $\bar\delta>0$ with the following property: if  $f_1,\, f_2 \in  C^1(\partial B^{(\mu)})$ with  $\|f_i\|_{C^1} \le \ov\delta$  and $|E_{f_i, \mu}|=\mu$ for $i=1,2$  we have
\begin{align}
&r(\mu)^2 (1-\eta) \frac{\|f_1 - f_2\|_{2}^2}{2} \leq \D(E_{f_1, \mu},E_{f_2, \mu}) \leq  r(\mu)^2 (1+\eta)\frac{\|f_1 - f_2\|_{2}^2}{2} ,\label{f1menof2}
\\
&\frac{1-\eta}2\int_{\partial E_{f_1, \mu}} d^2_{E_{f_2, \mu}} \, d\mathcal H^{N-1}\leq  \D(E_{f_1, \mu},E_{f_2, \mu}) \leq  \frac{ 1+\eta}2\int_{\partial E_{f_1, \mu}} d^2_{E_{f_2, \mu}} \, d\mathcal H^{N-1}\,,\label{disdis} \\ 
& 
|\bary( E_{f_1, \mu})-\bary( E_{f_2, \mu})|\leq C\sqrt{\D( E_{f_1, \mu},  E_{f_2, \mu})}\,,
\label{bari2000}
\end{align}
where $C>0$ depends only on $N$ and $\mu$.
\end{lemma}

\begin{proof}

We start by observing that for any $\eta'>0$, if $\bar\delta$ is sufficiently small, then for every $p_0\in \partial E_{f_2, \mu}$ we have
\begin{equation}\label{e:gono}
\partial E_{f_2, \mu} \cap B_{4\bar\delta}(p_0) \subset 
G:=\left\{y\in \R^N : \left\vert(y-p_0)\cdot \frac{p_0}{|p_0|}\right\vert^2\leq \frac{{\eta'}^2}{1+{\eta'}^2} |y-p_0|^2\right\}.
\end{equation}

We divide the rest of the proof into two steps.
 
\noindent{\bf Step 1.} If $\bar\delta$ is small enough, for every point $p=\lambda p_0\in B_{2\bar\delta}(p_0)$ ($\lambda>0$), we have that
\[
\frac1{1+\eta'}{|p-p_0|}\leq \textup{dist}(p, \partial E_{f_2,\mu}) \leq |p-p_0|.
\]
The second inequality is, indeed, obvious by definition, given $p_0\in \partial E_{f_2,\mu}$.
Concerning the first one, we notice that $\textup{dist}(p, \partial E_{f_2,\mu})\leq |p-p_0|\leq 2\bar\delta$ implies that there exists a point $q\in \partial E_{f_2,\mu}\cap B_{2\bar\delta}(p)$ such that
$\textup{dist}(p, \partial E_{f_2,\mu})= |p-q|$.
In particular, from \eqref{e:gono} we infer that $q\in G$ and hence we have 
\begin{align*}
\textup{dist}(p, \partial E_{f_2,\mu}) \geq \textup{dist}(p, G) = \frac1{\sqrt{1+{\eta'}^2}}|p-p_0|\geq \frac1{1+\eta'}{|p-p_0|}.
\end{align*}
In particular, if $p_0= (1+f_2(s))s\in \partial E_{f_2, \mu}$
with $s\in \partial B^{(\mu)}$ and 
\[
p_t:= p_0+t\,\frac{f_1(s)-f_2(s)}{|f_1(s)-f_2(s)|}\frac{s}{|s|} \quad\text{ for all }t\in \big[0,r(\mu)|f_1(s)-f_2(s)|\big],
\]
we deduce that
\begin{equation}\label{e:distanze}
\frac1{1+\eta'}\,t\leq \textup{dist}(p_t, \partial E_{f_2,\mu}) \leq t.
\end{equation}
Then, keeping this same notation and integrating in polar coordinates, we infer that
\begin{align}
\D(E_{f_1,\mu},E_{f_2,\mu})  &= 
\int_{E_{f_1,\mu}\Delta E_{f_2,\mu}}\mathrm{dist}(x, \partial E_{f_2,\mu})\, dx \notag\\
&= \int_{\partial B^{(\mu)}}  ds \int_{0}^{r(\mu)|f_2(s)-f_1(s)|} \mathrm{dist}(p_t, \partial E_{f_2,\mu}) \left( \frac{|p_t|}{r(\mu)}\right)^{N-1} \, dt.
\end{align}
Recalling that $\left\vert\frac{|p_t|}{r(\mu)} -  1\right\vert \leq \bar \delta$ and using \eqref{e:distanze}, we  get
\beq\label{split1}
\begin{split}
\D(E_{f_1,\mu},E_{f_2,\mu})  &\leq (1+\bar\de)^{N-1}
\int_{\partial B^{(\mu)}}  ds \int_{0}^{r(\mu)|f_2(s)-f_1(s)|} 
t \, dt\\ 
&= \frac{(1+\bar\de)^{N-1}}2r(\mu)^2 \int_{\partial B^{(\mu)}}  |f_1(s)-f_2(s)|^2 \,ds,
\end{split}
\eeq
from which the second inequality in \eqref{f1menof2} follows by taking $\bar \de$ smaller, if needed.
Analogously,
\beq\label{split2}
\begin{split}
\D(E_{f_1,\mu},E_{f_2,\mu})  &\geq \frac{(1-\bar\de)^{N-1}}{1+\eta'}
\int_{\partial B^{(\mu)}}  ds \int_{0}^{r(\mu)|f_2(s)-f_1(s)|} 
{t} \, dt\\ 
&= \frac{(1-\bar\de)^{N-1}}{2(1+\eta')} r(\mu)^2 \int_{\partial B^{(\mu)}}  |f_1(s)-f_2(s)|^2 \,ds\,,
\end{split}
\eeq
from which the first inequality in \eqref{f1menof2} follows by taking $\eta'$ and $\bar\de$ small enough.

\noindent{\bf Step 2.} The inequalities \eqref{disdis} and \eqref{bari2000} are now easy consequences. Indeed, by \eqref{e:distanze} we have that for every $x= (1+f_1(s))s\in \partial E_{f_1,\mu}$
\[
\frac{r(\mu)}{1+\eta'}|f_1(s) - f_2(s)|\leq d_{E_{f_2, \mu}}(x)\leq r(\mu)|f_1(s) - f_2(s)|.
\]
Therefore, \eqref{disdis} follows from \eqref{split1} and \eqref{split2}, by taking $\eta'$ and $\bar\de$ smaller if needed, through a simple change of coordinates (recall that the Jacobian of the map $s\mapsto (1+f_1(s))s$ and its inverse are estimated from above by $1+C\bar\delta$ for a suitable dimensional constant $C$).

Finally, note that we can write
$$
\bary( E_{f_i, \mu})= \frac{1}{(N+1)r(\mu)^{N}\omega_N}
\int_{\partial B^{(\mu)}} (1+f_i(s))^{N+1}s\,d\mathcal{H}^{N-1}(s)\,.
$$
Using the fact that $t\mapsto(1+t)^{N+1}$ is $2(N+1)$-Lipschitz for $t$ small, we may estimate

\begin{align*}
&|\bary( E_{f_1, \mu})-\bary( E_{f_2, \mu})| \\
&\leq 
\frac{1}{(N+1)r(\mu)^{N-1}\omega_N}\left\vert
\int_{\partial B^{(\mu)}} \big((1+f_1(s))^{N+1}-(1+f_2(s))^{N+1}\big)
d\mathcal{H}^{N-1}(s)
\right\vert\\
&\leq\frac{2}{r(\mu)^{N-1}\omega_N}
\int_{\partial B^{(\mu)}} |f_1(s)-f_2(s)| d\mathcal{H}^{N-1}(s)\leq C 
\|f_1(s)-f_2(s)\|_2\,,
\end{align*}
where $C>0$ depends only on $N$ and $\mu$. Thus \eqref{bari2000} follows from the previous inequality  combined with \eqref{f1menof2}.
\end{proof}

\begin{lemma}\label{disti}
Let $h>0$, $\mu>0$. There exists $C=C(h, \mu)>0$ and $\ov\de=\ov\de(h, \mu)>0$ with the following property: For any pair of sets  $E_{f_1, \mu}$, $E_{f_2, \mu}$ with $f_1 , \, f_2 \in C^2(\pa B^{(\mu)})$, 
$\|f_i\|_{C^1} \le \ov\delta$, 
and such that  $ |E_{f_2, \mu}|=\mu$, $\bary(E_{f_2, \mu})=0$ and 
\beq\label{eul}
H_{\pa E_{f_2, \mu}}+\frac{d_{E_{f_1, \mu}}}{h}=\lambda\qquad\text{on }\pa E_{f_2, \mu}
\eeq
 for some $\lambda\in \R$, we have
\beq\label{e:crucialdiss}
\D\big(B^{(\mu)}, E_{f_2, \mu}\big) \le  C \, \D(E_{f_2, \mu} , E_{f_1, \mu})\,.
\eeq
\end{lemma}

\begin{proof}

 By Theorem \ref{Aleq} (and Remark\til\ref{rm:acdc}), by choosing $\ov\de$ sufficiently small  and using also \eqref{eul}, we have
\beq\label{unoo}
 \begin{split}
 \|f_2\|^2_{L^2(\pa B^{(\mu )})}
& \leq C(\mu)
  \| H_{\partial E_{f_2, \mu}} - \overline H_{\partial E_{f_2, \mu}} \|^2_{L^2(\pa B^{(\mu)})}\leq C(\mu) \| H_{\partial E_{f_2, \mu}} - \lambda \|^2_{L^2(\pa B^{(\mu)})}\\
  &  \leq 2C(\mu) \| H_{\partial E_{f_2, \mu}} - \lambda \|^2_{L^2(\partial E_{f_2, \mu})}= \frac{2C(\mu)}{h^2}\int_{\partial E_{f_2, \mu}}d^2_{E_{f_1, \mu}}\, d\H^{N-1}\,,
 \end{split}
 \eeq
where the third inequality follows by bounding the Jacobian of the change of variables by $2$ (which can be done provided $\ov\de$ is small enough). 
Note now that by \eqref{f1menof2} (with $f_1$ replaced by $0$) and by \eqref{disdis} (and by taking $\ov\de$ smaller if needed) we have
$$
\D\big(B^{(\mu)}, E_{f_2, \mu}\big)\leq \|f_2\|^2_{L^2(\pa B^{(\mu )})}
\qquad\text{and}\qquad
\int_{\partial E_{f_2, \mu}}d^2_{E_{f_1, \mu}}\, d\H^{N-1}\leq 4 \D(E_{f_2, \mu} , E_{f_1, \mu})\,.
$$
Combining the previous inequalities with \eqref{unoo}, the conclusion follows.
  \end{proof}

\begin{remark}\label{rm:purdinonfinire}
It is clear that  the constants $C$ and $\ov\de$ in Lemmas\til\ref{disti0} and \ref{disti} are uniform with respect to $\mu$ varying on any compact subset of $(0,+\infty)$.
\end{remark}

%
We are now ready to prove the main result of the paper. The main difficulty is in controlling the translations introduced in Proposition\til\ref{prop:uptrans} and in proving the convergence of the barycenters. A crucial role in such an  argument is played by the  dissipation/dissipation inequality \eqref{e:crucialdiss}, which in turn relies on the quantitative Alexandrov type estimate established in Theorem\til\ref{Aleq}.

\begin{proof}[Proof of Theorem\til\ref{mainthm}]
We split the proof into several steps.

\noindent{\bf Step 1.} (Exponential decay of dissipations) 
Recall that by Proposition\til\ref{prop:uptrans} we have
$$
P(E^n_h)\to LP\big(B^{\left(\frac{m}L\right)}\big)\,.
$$
Thus,  summing \eqref{triviale} from $n+1\in \N$ to $+\infty$, we get

\begin{equation}\label{eq1}
\sum_{k=n+1}^{+\infty}  \frac 1h \D(E_h^k,E_{h}^{k-1}) \le  P(E_h^{n}) - LP\Big(B^{\left(\frac{m}L\right)}\Big).
\end{equation}
Recall that again by Proposition\til\ref{prop:uptrans} for $n$ large enough each set $E_h^n$ is made up of $L$ connected components $E_{h, 1}^n, \dots,  E_{h, L}^n$,  
\beq\label{rec1}
m^n_{i}:=|E_{h, i}^n|\to \frac{m}L 
\eeq
and 
\beq\label{rec2}
E_{h, i}^n-\xi_i^n\to B^{\left(\frac{m}L\right)}\qquad\text{in }C^k
\eeq
as $n\to\infty$, where we set
$$
\xi_i^n:=\bary(E_{h, i}^n)\,.
$$
With Lemma\til\ref{uniconvd} in mind, we also get
\beq\label{rec2.3}
E_{h, i}^{n-1}-\xi_i^n\to B^{\left(\frac{m}L\right)}\qquad\text{in }C^k
\eeq
as $n\to\infty$.
Combining \eqref{rec2} and \eqref{rec2.3}, we have that for any $k\in \N$ and for $n$ large enough there exist functions $f_{1,i}^n$, $f_{2,i}^{n}\in C^k(\pa B^{(m^n_i)})$ such that (with the notation introduced in \eqref{schifo})
\beq\label{rec2.5}
E_{h, i}^n-\xi_i^n=E_{f_{2,i}^n, m_i^n},\, \,
E_{h, i}^{n-1}-\xi_i^n=E_{f_{1,i}^n, m_i^n}
\quad\text{with }\|f_{1,i}^n\|_{C^k},
\|f_{2,i}^n\|_{C^k} \to 0\text{ as }n\to\infty\,.
\eeq
Moreover, again by Proposition\til\ref{prop:uptrans} for $n$ large enough we have
\beq\label{rec3}
\dist\left(E_{h, i}^n, E_{h, j}^n\right)\geq\frac{s_0}2\,,   \qquad\text{ for }i\neq j\,, 
\eeq
with $s_0$ the constant of Proposition\til\ref{fpds}.
Consider now  the the admissible competitor for $E_h^n$ given by
$$
\mathfrak{B}_n:=\bigcup_{i=1}^L\left(\xi_{ i}^{n-1}+B^{(m^{n-1}_{i})}\right)\,,
$$
and note that by \eqref{rec1} and \eqref{rec2} we also have
$$
\dist\left(\xi_{ i}^{n-1}+B^{(m^{n-1}_{i})}, \xi_{ j}^{n-1}+B^{(m^{n-1}_{j})}\right)\geq\frac{s_0}4
$$
for $n$ large enough and $i\neq j$. The above inequality and \eqref{rec3} in turn yield that 
\beq\label{rec4}
\begin{array}{lcl}
\D(E_h^n,E_{h}^{n-1})&=&\displaystyle\sum_{i=1}^{L}\D(E_{h, i}^n,E_{h, i}^{n-1})\text{ and }\\ 
\D(\mathfrak{B}_n,E_{h}^{n-1})&=&\displaystyle
\sum_{i=1}^{L}\D\left(B^{(m^{n-1}_i)},E_{h, i}^{n-1}-\xi_i^{n-1})\right)\,.
\end{array}
\eeq
Testing the minimality of $E_{h, i}^{n}$ with $\mathfrak{B}_n$ and using the second identity in \eqref{rec4}, we have 
\beq\label{rec5}
P(E_h^n) + \frac 1h \D(E_h^n,E_{h}^{n-1}) \le P(\mathfrak{B_n}) + \frac 1h \sum_{i=1}^{L}\D\left(B^{(m^{n-1}_i)},E_{h, i}^{n-1}-\xi_i^{n-1}\right)\,.
\eeq
Recall now that by \eqref{rec2.5} and by Lemma\til\ref{disti} (see also Remark\til\ref{rm:purdinonfinire}) for $n$ large enough we have
\begin{multline*}
\D\left(B^{(m^{n-1}_i)},E_{h, i}^{n-1}-\xi_i^{n-1}\right)=
\D\left(B^{(m^{n-1}_i)},E_{f_{2, m^{n-1}_i}}\right)\\
\leq C
\D\left(E_{f_{2, m^{n-1}_i}},E_{f_{1, m^{n-1}_i}}\right)=C\D(E_{h, i}^{n-1}, E_{h, i}^{n-2})\,.
\end{multline*}
Thus,  from \eqref{rec5}  and \eqref{rec4} we deduce that
\beq\label{rec6}
P(E_h^n)-P(\mathfrak{B_n})\leq \frac Ch \sum_{i=1}^{L}\D(E_{h, i}^{n-1}, E_{h, i}^{n-2})=\frac Ch \D(E_{h}^{n-1}, E_{h}^{n-2})\,.
\eeq
Observe now that by  concavity 
$$
\sum_{i=1}^L m_i^{\frac{N-1}{N}}\leq L \left(\frac mL\right)^{\frac{N-1}{N}}\quad \text{for all } m_1, \dots, m_L\geq 0 \text{ s.t. }\sum_{i=1}^Lm_i=m
$$
and thus
$$
P(\mathfrak{B_n})\leq LP\big(B^{\left(\frac{m}L\right)}\big)\,.
$$
Therefore, from \eqref{rec6} and \eqref{eq1} we get
\begin{align*}
&\sum_{k=n-1}^{+\infty}  \frac 1h \D(E_h^k,E_{h}^{k-1})\\
&= \sum_{k=n+1}^{+\infty}  \frac 1h \D(E_h^k,E_{h}^{k-1})+ 
 \frac 1h \D(E_h^{n-1},E_{h}^{n-2})+\frac 1h \D(E_h^{n},E_{h}^{n-1})\\
 &\leq   P(E_h^{n}) - LP\Big(B^{\left(\frac{m}L\right)}\Big)+\frac 1h \D(E_h^{n-1},E_{h}^{n-2})+\frac 1h \D(E_h^{n},E_{h}^{n-1})\\
 &\leq \frac {C+1}h \D(E_{h}^{n-1}, E_{h}^{n-2})+\frac 1h \D(E_h^{n},E_{h}^{n-1})\leq \frac{C+1}{h}\Big(\D(E_{h}^{n-1}, E_{h}^{n-2})+ \D(E_h^{n},E_{h}^{n-1})\Big)\,.
\end{align*}

 We may now apply Lemma \ref{an1} (with $\ell=2$) below to conclude
\beq\label{finalmente}
\D(E_h^n,E_{h}^{n-1})\leq \left(1-\frac{1}{C+1}\right)^{\frac n2}\left(P(E) - LP(B^{(\frac{m}L)})\right)\,.
\eeq

\noindent {\bf Step 2.} (Exponential convergence of the barycenters)
By \eqref{rec2.5}, \eqref{bari2000} and by \eqref{finalmente}, setting
\beq\label{b}
b:= \left(1-\frac{1}{C+1}\right)^{\frac 14}\in (0,1)\,,
\eeq
we have for $n$ sufficiently large
\begin{align*}
|\xi^{n}_i-\xi_i^{n-1}|&=\big|\bary(E_{f_{2,i}^n, m^i_n})-\bary(E_{f_{1,i}^n, m^i_n})\big|\\
&\leq C\sqrt{\D\big(E_{f_{2,i}^n, m^i_n}, E_{f_{1,i}^n, m^i_n}\big) }=
C\sqrt{\D(E^{n}_{h,i},  E_{h,i}^{n-1}) }\\
&\leq C \left(P(E) - LP(B^{(\frac{m}L)})\right)^{\frac12}b^n\,.
\end{align*}
In turn, the above estimate implies that $\{\xi_i^n\}_n$ satisfies the Cauchy condition and thus there exist $\xi^\infty_i\in \R^N$, $i=1, \dots, L$, such that  $\xi_i^n\to \xi_i^{\infty}$ exponentially fast as $n\to\infty$; precisely,
$$
|\xi_i^n-\xi_i^{\infty}|\leq \sum_{k=n+1}^{\infty}|\xi_k^n-\xi_{k-1}^n|\leq C \left(P(E) - LP(B^{(\frac{m}L)})\right)^{\frac12}\frac{b^{n}}{1-b}\,
$$
for $n$ large enough and for $i=1, \dots, L$. Recalling \eqref{rec2}, we may conclude that for all $k\in \N$
\beq\label{finalmente2}
E^n_{h,i}\to \xi^\infty_i+B^{(\frac mL)}\quad\text{ in }C^k \quad\text{as }n\to\infty \text{ and for }i=1, \dots, L\,.
\eeq

\noindent{\bf Step 3.} (Exponential convergence of the sets) By \eqref{finalmente2} we can parametrize the boudaries of the sets $E^n_{h,i}-\xi_i^\infty$ as radial graphs over the limiting ball
$B^{(\frac mL)}$. Precisely, again with the notation \eqref{schifo}, there exist functions 
$g_i^n$ such that 
\beq\label{tuttob}
E^n_{h,i}-\xi_i^\infty=E_{g_i^n, \frac mL}\qquad\text{and}\qquad \|g_i^n\|_{C^k\big(\pa B^{(\frac mL)}\big)}\to 0\text{ as }n\to\infty\,.
\eeq
In turn, by Lemma\til\ref{disti0} (see \eqref{f1menof2}), for $n$ large enough we have that 
$\|g_i^n-g_i^{n-1}\|_{L^2(\pa B^{(\frac mL)} )}\leq 2\sqrt{\D(E^n_{h,i}, E^{n-1}_{h,i})}$ and, thus, recalling \eqref{finalmente} and arguing as in Step 2, we get
\beq\label{expl2}
\|g_i^n\|_2\leq \sum_{k=n+1}^{\infty}\|g_i^k-g_i^{k-1}\|_2\leq 
2\sum_{k=n+1}^{\infty}\sqrt{\D(E^{k}_{h,i}, E^{k-1}_{h,i})}\leq 
\left(P(E) - LP(B^{(\frac{m}L)})\right)^{\frac12}\frac{b^n}{1-b}\,,
\eeq
where $b$ is as in \eqref{b}. The above estimate yields the exponential decay of the $L^2$-norms of the  radial graphs. We now recall the following well known interpolation inequality:  for every $j\in \N$ there exists $C>0$  such that if $g$ sufficiently smooth on
 $\pa B^{(\frac mL)}$, then
\beq\label{interp}
\|D^k g\|_{L^2\big(\pa B^{(\frac mL)}\big)}\leq
 C \|D^{2k} g\|^{\frac12}_{L^2\big(\pa B^{(\frac mL)}\big)}\|g\|^{\frac12}_{L^2\big(\pa B^{(\frac mL)}\big)}\,,
\eeq
where $D^k$ stands for the collection all $k$-th order (covariant) derivatives of $g$, see for instance \cite{Aubin}. Now, using the fact that from \eqref{tuttob} for every $k$ there exists $n_k\in \N$ such that 
$$
\sup_{n\geq n_k}\|D^{2k} g^n_i\|_{2}\leq 1\,,
$$ 
we may apply \eqref{interp} to $g^{n}_i$ to deduce from \eqref{expl2} that also 
$\|D^{k} g^n_i\|_{2}$ decays exponentially fast for all $k\in \N$. This in turn yields the exponential decay in $C^k$ for every $k$ and concludes the proof of the theorem.
\end{proof}

\begin{lemma}\label{an1}
Let $\{a_n\}_{n\in \mathbb{N}}$ be  a sequence of non-negative numbers.
Assume furthermore that there exists $c>1$, $\ell\in \N$ such that $\sum_{n=k}^{+\infty}a_n\leq c\,\sum_{j=k}^{k+\ell-1}a_j$ for every $k\in\N$. Then,
$$
a_k\leq \Bigl(1-\frac1c\Bigr)^{\frac{k}{\ell}}S
$$
for every $k\in \N$, where $S:=\sum_{n=1}^{+\infty}a_n$.
\end{lemma}
\begin{proof}
We first consider the case  $\ell=1$. 
Set $F(k):=\sum_{n=k}^{+\infty}a_n$ and note that  by assumption $F(k)\leq c(F(k)-F(k+1))$ for every $k\in \N$. Hence, by iteration we have
$$
a_{k+1}\leq F(k+1)\leq  \Bigl(1-\frac1c\Bigr)F(k)\leq\dots\leq  \Bigl(1-\frac1c\Bigr)^{k+1}F(0)= \Bigl(1-\frac1c\Bigr)^{k+1} S
$$
 for every $k\in \N$. 
 
In the case $\ell\ge 2$ it is enough to set $b_k:=\sum_{j=1}^{\ell} a_{\ell(k-1)+j}$  and to observe that the assumption now reads 
$\sum_{n=k}^{+\infty}b_n\leq c b_k$ so that we may apply the previous case.\end{proof}

\section*{Acknowledgements}
This research was partially supported by the GNAMPA Grant 2018 ``Nonlocal geometric flows'' and by the GNAMPA Grant 2019 ``Variational methods for nonlocal gemetric motions'', both funded by INDAM.

E.~S.~has been supported by the ERC-STG Grant n. 759229
HiCoS ``Higher Co-dimension Singularities: Minimal Surfaces and 
the Thin Obstacle Problem''.

\end{document}